\RequirePackage{fix-cm}
\documentclass[smallextended]{svjour3}       
\smartqed  
\usepackage{graphicx}
\usepackage{xcolor}
\newtheorem{algorithm}{Algorithm}
\journalname{}
\begin{document}

\title{Parallel hybrid extragradient methods for pseudomonotone equilibrium problems and nonexpansive mappings}


\titlerunning{Parallel hybrid extragradient methods}        

\author{Dang Van Hieu         \and
        Le Dung Muu \and Pham Ky Anh}

\authorrunning{D.V. Hieu, L.D. Muu, and P.K. Anh} 
\institute{Dang Van Hieu \at
              Department of Mathematics, Vietnam National University,  \\
              334 Nguyen Trai, Thanh Xuan, Hanoi, Vietnam \\
              \email{dv.hieu83@gmail.com}           
           \and
           Le Dung Muu \at
              Institute of Mathematics, VAST, Hanoi, \\
             18 Hoang Quoc Viet, Hanoi, Vietnam \\
               \email{ldmuu@math.ac.vn}
             \and
             Pham Ky Anh \at
             Department of Mathematics, Vietnam National University, Hanoi,  \\
              334 Nguyen Trai, Thanh Xuan, Hanoi, Vietnam \\
             \email{anhpk@vnu.edu.vn}
}

%
%
\date{Received: date / Accepted: date}

\maketitle

\begin{abstract}
In this paper we propose and analyze three parallel hybrid extragradient methods for finding a common element of the set of  solutions of 
equilibrium problems involving pseudomonotone bifunctions and the set of fixed
points of nonexpansive mappings in a real Hilbert space. Based on parallel computation we can reduce 
the overall computational effort under widely used conditions on the bifunctions and the nonexpansive mappings.
A simple numerical example is given to illustrate the proposed parallel algorithms.
\keywords{Equilibrium problem\and Pseudomonotone bifunction\and Lipschitz-type continuity\and Nonexpansive mapping\and Hybrid 
method\and Parallel computation}
\end{abstract}

\section{Introduction}\label{intro}
Let $C $ be a nonempty closed convex subset of a real Hilbert space $H.$ The equilibrium problem for a bifunction 
$f: C \times C \to \Re  \cup \{+ \infty \},$ satisfying condition $f(x,x) = 0$ for every $x \in C$, is stated as follows:
\begin{equation}\label{EP}
{\rm Find}\,\, x^* \in C\,\, {\rm such \,that}\quad f(x^*,y)\ge 0 \quad \forall y\in C.
\end{equation}
The set of solutions of $(\ref{EP})$ is denoted by $EP(f)$. Problem $(\ref{EP})$ includes, as special cases, many mathematical models, 
such as, 
optimization problems, saddle point problems, Nash equilibrium point problems, fixed point problems, convex differentiable optimization 
problems, variational inequalities, complementarity problems, etc., see \cite{BO1994,MO1992}. In recent years, many methods have been proposed 
for 
solving equilibrium problems, for instance, see \cite{KSZ2010,SK2012,SLZ2011,TT2007} and the references therein.\\
A mapping $T:C\to C$ is said to be nonexpansive if $||T(x)-T(y)||\le ||x-y||$ for all $x,y\in C$. The set of fixed points of $T$ is denoted by  $F(T)$. \\
Finding common elements  of the solution set of an equilibrium problem and the fixed point set of a nonexpansive mapping is a task arising frequently in various areas of mathematical sciences, engineering, and economy. For example, we consider the following extension of  a Nash-Cournot oligopolistic equilibrium model  \cite{FP2002}.\\
 Assume that there are $n$ companies that produce a commodity. Let $x$
denote the vector whose entry $x_j$ stands for the quantity  of
 the commodity producing  by company $j$. We suppose that the price $p_i(s)$ is a
decreasing affine function of $s$ with $s= \sum_{j=1}^n x_j$, i.e., $p_i(s)= \alpha_i - \beta_i s$, where $\alpha_i > 0$, $\beta_i > 0$. Then
the profit made by company $j$ is given by $f_j(x)= p_j(s) x_j) -
c_j( x_j)$, where $c_j(x_j)$ is the tax for generating $x_j$.
Suppose that $K_j$ is the strategy set of company $j$,  Then the
strategy set of the model is $K:= K_1\times\times ...\times K_n$.
Actually, each company seeks to  maximize its profit by choosing the
corresponding production level under the presumption that  the
production of the other companies is a parametric input.
 A commonly used approach to this model is based upon the famous Nash
   equilibrium concept.\\
\indent We recall  that a point $x^* \in K=K_1\times
K_2\times\cdots\times K_n$ is  an
    equilibrium point of the model   if
    $$f_j(x^*) \geq f_j(x^*[x_j]) \ \forall x_j \in K_j, \ \forall  j=1,2,\ldots,n,$$
    where the vector $x^*[x_j]$ stands for the vector obtained from
    $x^*$ by replacing $x^*_j$ with $x_j$.
By taking
$$f(x, y):= \psi(x,y)-\psi(x,x)$$
with
\begin{equation}\label {pep}
\psi(x,y):=  -\sum_{j = 1}^n f_j(x[y_j]),
\end{equation}
the problem of finding a Nash equilibrium point of the model can be
formulated as
$$x^* \in K: f(x^*,x) \geq 0 \ \forall x \in K. \eqno(EP)$$
 In practice each company has to pay a fee $g_j(x_j)$ depending on its production level
 $x_j$.\\
The problem now is to find an equilibrium point with minimum fee.
We suppose that  both  tax and   fee functions   are convex for every $j$. The convexity assumption    means that the  tax and fee  for producing a unit are increasing as the quantity of the production gets larger.
The convex assumption  on $c_j$ implies that   the bifunction $f$ is monotone on $K$, while the convex assumption on $g_j$ ensures that the solution-set of the convex  problem $$\min\{g(x) =
 \sum_{j-1}^n g_j(x_j): x\in K\}$$
  coincides with  fixed point-set of the nonexpansive  proximal
  operator $P:= (I +c\partial g)^{-1}$ with $c > 0$  \cite{R1976}.\\
Thus the problem of finding an equilibrium point with minimal cost
 is actually of the same kind as the problem studied in this paper.\\
Gradient based methods dealing with equilibrium problems as well as iteration methods for nonexpansive and 
pseudocontractive mappings have been studied by several authors ( see, \cite{YPL2013,YPL2013a,TP2013,YP2012,YPL2013} and 
the references therein).\\ 
For finding a common element of the set of solutions of  monotone equilibrium problem $(\ref{EP})$ 
 and the set of fixed points of  a nonexpansive mapping $T$ in Hilbert spaces, Tada and Takahashi \cite{TT2006} proposed the following 
hybrid method:
$$
\left \{
\begin{array}{ll}
&x_0\in C_0=Q_0=C,\\ 
&z_n \in C\quad \mbox{such that}\quad f(z_n,y)+\frac{1}{\lambda_n}\left\langle y-z_n,z_n-x_n \right\rangle \ge 0, \forall y\in C,\\
&w_n=\alpha_n x_n+(1-\alpha_n)T(z_n),\\
&C_n = \{v \in C: ||w_n - v||\leq ||x_n-v||\},\\ 
& Q_n = \{ v \in C:  \langle x_0 - x_n, v - x_n \rangle \leq 0 \},\\
&x_{n+1}= P_{C_n \cap Q_n}(x_0).
\end{array}
\right.
$$
According to the above algorithm, at each step for determining the intermediate approximation $z_n$ we need to solve a strongly 
monotone regularized 
equilibrium problem
\begin{equation}\label{eq:Re.subproblem}
\mbox{Find}\,z _n \in C, ~ \mbox{such that}\, f(z_n,y)+\frac{1}{\lambda_n}\left\langle y-z_n,z_n-x_n \right\rangle \ge 0, ~ \forall y\in C.
\end{equation}
If the bifunction $f$ is only pseudomonotone,  then subproblem $(\ref{eq:Re.subproblem})$ is not necessarily
strongly monotone, even not pseudomonotone, hence the existing algorithms using the monotonicity of the subproblem, cannot be applied. 
To overcome this difficulty, Anh \cite{A2011} proposed the following hybrid 
extragradient method for finding a common element of the set of fixed points of a nonexpansive mapping $T$ and the set of solutions of 
an equilibrium problem involving a pseudomonotone bifunction $f$.
$$
\left \{
\begin{array}{ll}
&x_0 \in C,C_0=Q_0=C,\\ 
&y_n= \arg\min \{ \lambda_n f(x_n, y) +\frac{1}{2}||x_n-y||^2:  y \in C\},\\
&t_n = \arg\min \{ \lambda_n f(y_n, y) +\frac{1}{2}||x_n-y||^2:  y \in C\},\\
&z_n=\alpha_n x_n+(1-\alpha_n)T(t_n),\\
&C_n = \{v \in C: ||z_n - v||\leq ||x_n-v||\},\\ 
& Q_n = \{ v \in C:  \langle x_0 - x_n, v - x_n \rangle \leq 0 \},\\
&x_{n+1}= P_{C_n \cap Q_n}(x_0).
\end{array}
\right.
$$
Under certain assumptions, the strong convergence of the sequences 
$\left\{x_n\right\}$, $\left\{y_n\right\}$, $\left\{z_n\right\}$ to $x^\dagger:=P_{EP(f)\cap F(T)}x_0$ has been established. \\
Very recently, Anh and Chung \cite{AC2013} have proposed the following parallel hybrid method for finding a common fixed point of a 
finite family of relatively nonexpansive mappings $\left\{T_i\right\}_{i=1}^N.$ 
\begin{equation}\label{Algor.AC2014}
\left \{
\begin{array}{ll}
&x_0 \in C,C_0=Q_0=C,\\ 
&y_n^i=J^{-1}\left(\alpha_n Jx_n+(1-\alpha_n)JT_i(x_n)\right),i=1,\ldots,N,\\
&i_n=\arg\max_{1\le i\le N}\left\{\left\|y_n^i-x_n \right\|\right\}, \quad  \bar{y}_n := y_n^{i_n},\\
&C_n=\left\{v\in C:\phi(v,\bar{y}_n)\le \phi(v,x_n)\right\},\\
&Q_n=\left\{v\in C : \left\langle Jx_0-Jx_n,x_n-v\right\rangle \ge 0\right\},\\
&x_{n+1}=P_{C_n\bigcap Q_n}x_0,n\ge 0,
\end{array}
\right.
\end{equation}
where $J$ is the normalized duality mapping and $\phi(x,y)$ is the Lyapunov functional. 
This algorithm was extended, modified and generelized by Anh and Hieu \cite{AH2014} for a finite family of asymptotically quasi 
$\phi$-nonexpansive mappings in Banach spaces. \\
According to algorithm $(\ref{Algor.AC2014})$, the intermediate approximations $y_n^i$ can be found in parallel. Then the farthest 
element from $x_n$ among all 
$y_n^i,  i=1,\ldots,N,$ denoted by $\bar{y}_n$, is chosen. Using the element $\bar{y}_n,$ the authors constructed two convex 
closed subsets $C_n$ and $Q_n$ containing the set of common fixed points $F$ and seperating the initial approximation $x_0$ from $F$. 
The next approximation $x_{n+1}$ is defined as the projection of $x_0$ onto the intersection $C_n\bigcap Q_n$. \\
The purpose of this paper is to propose three parallel hybrid extragradient algorithms for finding a common 
element of the set of solutions 
of a finite family of equilibrium problems for pseudomonotone bifunctions $\left\{f_i\right\}_{i=1}^N$ and the set of fixed points of a 
finite family of nonexpansive mappings $\left\{S_j\right\}_{j=1}^M$ in Hilbert spaces. We combine the extragradient method for dealing 
with pseudomonotone equilibrium problems (see, \cite{A2011,QMH2008}), and Mann's or Halpern's
iterative algorithms for finding fixed points of 
nonexpansive mappings \cite{H1967,M1953}, with parallel splitting-up techniques \cite{AC2013,AH2014}, as well as hybrid methods 
(see, \cite{A2011,AC2013,AH2014,KSZ2010,PU2008,SK2012,SLZ2011}) to obtain the strong convergence of iterative processes.\\
The paper is organized as follows: In Section 2, we recall some definitions and preliminary results. Section 3 deals with novel 
parallel hybrid algorithms and their convergence analysis. Finally, in Section 4, we  illustrate the propesed parallel hybrid 
methods by considering  a simple numerical experiment.
\section{Preliminaries}\label{Preli}
In this section, we recall some definitions and results that will be used in the sequel. 
Let $C$ be a nonempty closed convex subset of a Hilbert space 
$H$ with an inner product $\left\langle .,.\right\rangle$ and the induced norm $||.||$. Let $T:C\to C$ be a nonexpansive mapping with the 
set of fixed points $F(T)$.\\
We begin with the following properties of nonexpansive mappings.
\begin{lemma}\cite{GK1990}\label{lem.demiclose}
Assume that $T:C\to C$ is a nonexpansive mapping. If $T$ has a fixed point, then 
\begin{enumerate}
\item [$(i)$] $F(T)$ is a closed convex subset of $H$.
\item [$(ii)$] $I-T$ is demiclosed, i.e., whenever $\left\{x_n\right\}$ is a sequence in $C$ weakly converging to some $x\in C$ and the 
sequence $\left\{(I-T)x_n\right\}$ strongly converges to some $y$, it follows that $(I-T)x=y$.
\end{enumerate}
\end{lemma}
\noindent Since $C$ is a nonempty closed and convex subset of $H$, for every $x\in H$, there exists a unique element $P_C x,$ defined by
$$
P_C x=\arg\min\left\{\left\|y-x\right\|:y\in C\right\}.
$$
The mapping $P_C:H\to C$ is called the metric (orthogonal) projection of $H$ onto $C$. It is also known that $P_C$ is firmly 
nonexpansive, or 
$1$-inverse strongly monotone ($1$-ism), i.e., 
$$
\left\langle P_C x-P_C y,x-y \right\rangle \ge \left\|P_C x-P_C y\right\|^2.
$$
Besides, we have
\begin{equation}\label{eq:ProperOfPC}
\left\|x-P_C y\right\|^2+\left\|P_C y-y\right\|^2\le \left\|x-y\right\|^2.
\end{equation}
Moreover, $z=P_C x$ if and only if 
\begin{equation}\label{eq:EquivalentPC}
\left\langle x-z,z-y \right\rangle \ge 0,\quad \forall y\in C.
\end{equation}
A function $f: C \times C \to \Re \cup \{+ \infty \},$ where $C \subset H$ is a closed convex 
subset, such that $f(x,x) = 0$ for all $x \in C$  is called a bifunction. Throughout this paper we consider bifunctions with 
the following 
properties: 
\begin{itemize}
\item[A1.] $f$ is pseudomonotone, i.e., for all $x,y\in C$,
$$ f(x, y) \geq 0 \Rightarrow  f(y,x) \leq 0; $$
\item [A2.]  $f$ is Lipschitz-type continuous, i.e., there exist two positive constants $c_1,c_2$ such that
$$ f(x,y) + f(y,z) \geq f(x,z) - c_1||x-y||^2 - c_2||y-z||^2, \quad \forall x,y,z \in C;$$
\item [A3.]   $f$ is weakly continuous on $C\times C$;
\item [A4.]  $f(x,.)$ is convex and subdifferentiable on $C$  for every fixed $x\in C.$
\end{itemize}  
A bifunction $f$ is called monotone on $C$ if for all $x,y\in C, \quad f(x,y)+f(y,x)\le 0.$  It is obvious that any monotone bifunction  is
a pseudomonotone one, but not vice versa. Recall that a mapping $A:C\to H$ is pseudomonotone if and only if the bifunction 
$f(x,y)=\left\langle A(x), y-x\right\rangle$ is pseudomonotone on $C.$\\
The following statements will be needed in the next section.
\begin{lemma}\label{EPexistence}\cite{BS1996}
If the bifunction $f$ satisfies Assumptions $A1-A4$, then the solution set $EP(f)$ is weakly closed and convex.
\end{lemma}
\begin{lemma}\cite{DGM2003}\label{lem.Equivalent_MinPro}
Let $C$ be a convex subset of a real Hilbert space H and $g:C\to \Re$ be a convex and subdifferentiable function on $C$. Then, 
$x^*$ is a solution to the following convex problem
$$
\min\left\{g(x):x\in C\right\}
$$
if and only if  ~  $0\in \partial g(x^*)+N_C(x^*)$, where $\partial g(.)$ denotes the subdifferential of $g$ and $N_C(x^*)$ is the normal cone 
of $C$ at $x^*$.
\end{lemma}
\begin{lemma}\cite{PU2008}\label{lem:PU2008}
Let $X$ be a uniformly convex Banach space, $r$ be a positive number and $B_r(0) \subset X $ be a closed ball with center at origin and the
radius $r$. Then, for any given subset 
$\left\{x_1,x_2,\ldots,x_N\right\}\subset B_r(0)$ and for any positive numbers $\lambda_1,\lambda_2,\ldots,\lambda_N$ with 
$\sum_{i=1}^N\lambda_i=1$, there exists  a continuous, strictly increasing, and convex function $g:[0,2r)\to [0,\infty)$ with $g(0)=0$ 
such that, for any $i,j\in \left\{1,2,\ldots,N\right\}$ with $i<j$,
$$
\left\|\sum_{k=1}^N\lambda_kx_k\right\|^2\le\sum_{k=1}^N \lambda_k \left\|x_k \right\|^2-\lambda_i\lambda_j g(||x_i-x_j||).
$$
\end{lemma}
\section{Main results}
In this section, we propose three novel parallel hybrid extragradient algorithms for finding a common element of the set of solutions of 
equilibrium problems for pseudomonotone bifunctions $\left\{f_i\right\}_{i=1}^N$ and the set of fixed points of  nonexpansive 
mappings $\left\{S_j\right\}_{j=1}^M$ in a real Hilbert space $H$. \\
In what follows, we assume that the solution set $$F=\left(\cap_{i=1}^N EP(f_i)\right)\bigcap \left(\cap_{j=1}^M F(S_j)\right)$$ is 
nonempty and each bifunction $f_i ~ (i=1,\ldots, N)$ satisfies all the conditions $A1-A4.$ \\
Observe that we can choose the same 
Lipschitz coefficients $\{c_1, c_2\}$ for all bifunctions $f_i,   i=1,\ldots,N.$ Indeed, condition $A2$ implies that
$f_i(x,z) - f_i(x,y) - f_i(y,z) \leq c_{1,i}||x-y||^2 + c_{2,i}||y-z||^2 \leq c_1||x-y||^2 + c_2||y-z||^2, $ where  $c_1 = \max \left\{c_{1,i}:i=1,\ldots,N\right\}$ and 
$c_2 =\max \left\{c_{2,i}:i=1,\ldots,N\right\}.$ Hence,  $ f_i(x,y) + f_i(y,z) \geq f_i(x,z) - c_1||x-y||^2 - c_2||y-z||^2.$\\
Further, since $F \ne \emptyset$, by Lemmas \ref{lem.demiclose}, \ref{EPexistence}, the sets $F(S_j) ~ j=1,\ldots, M$  and $EP(f_i) ~  i=1,\ldots, N$ are nonempty, 
closed and convex, hence
the solution set $F$ is a nonempty closed and convex subset of $C$. Thus,  given any fixed element $x^0 \in C$ there exists a unique 
element $x^\dagger := P_F(x^0)$.
\begin{algorithm}\label{algorithm1}{\rm (Parallel Hybrid Mann-extragradient method)}\\
\noindent \textbf{Initialization} $x^0 \in C, 0< \rho <\min \left(\frac{1}{2c_1},\frac{1}{2c_2}\right),~  n := 0$ and the sequence 
$\left\{\alpha_k\right\} 
\subset (0,1)$ satisfies the condition $\lim\sup_{k\to\infty}\alpha_k<1$.\\
\textbf{Step 1.} Solve $N$ strongly convex programs in parallel
$$y_n^i = {\rm argmin} \{ \rho f_i(x_n, y) +\frac{1}{2}||x_n-y||^2:  y \in C\} \quad i=1,\ldots,N.$$
\textbf{Step 2.} Solve $N$ strongly convex programs in parallel
 $$ z_n^i = {\rm argmin} \{ \rho f_i(y_n^i, y) +\frac{1}{2}||x_n-y||^2:  y \in C\} \quad i=1,\ldots,N. $$
\textbf{Step 3.} Find among $z_n^i, \quad i =1,\ldots,N,$ the farthest element from $x_n$, i.e., 
$$ i_n = {\rm argmax}\{||z_n^i - x_n||: i =1,\ldots,N\},\bar{z}_n:=z^{i_n}_n. $$
\textbf{Step 4.} Find intermediate approximations $u_n^j$ in parallel
$$ u_n^j=\alpha_n x_n+(1-\alpha_n)S_j \bar{z}_n, j=1,\ldots,M. $$
\textbf{Step 5.} Find among $u_n^j, \quad j =1,\ldots,M,$ the farthest element from $x_n$, i.e., 
$$ j_n= {\rm argmax}\{||u_n^j - x_n||: j =1,\ldots,M\},\bar{u}_n:=u^{j_n}_n. $$
\textbf{Step 6.} Construct two closed convex subsets of C
\begin{eqnarray*}
&&C_n = \{v \in C: ||\bar{u}_n - v||\leq ||x_n-v||\},\\ 
&&  Q_n = \{ v \in C:  \langle x_0 - x_n, v - x_n \rangle \leq 0 \}.
\end{eqnarray*}
\textbf{Step 7.} The next approximation $x_{n+1}$ is defined as the projection of $x_0$ onto $C_n \cap Q_n$, i.e.,
$$ x_{n+1}= P_{C_n \cap Q_n}(x_0). $$
\textbf{Step 8.} If $x_{n+1}=x_n$ then stop. Otherwise, set $n:=n+1$ and go to \textbf{Step 1}.
\end{algorithm}
For establishing the strong convergence of  Algorithm $\ref{algorithm1}$, we need the following results. 
\begin{lemma}\label{lem.help1}\cite{A2011,QMH2008}
Suppose that $x^*\in EP(f_i) ,$ and $x_n,  y^i_n, z^i_n, ~  i=1,\ldots, N, $ are defined as in Step 1 and  Step 2 of Algorithm $\ref{algorithm1}$. Then
\begin{equation}\label{eq:1}
||z_n^i-x^*||^2\le ||x_n-x^*||^2-(1-2\rho c_1)||y_n^i-x_n||^2-(1-2\rho c_2)||y_n^i-z_n^i||^2.
\end{equation}
\end{lemma}
\begin{lemma}\label{lem:CQ_closedconvex}
If Algorithm $\ref{algorithm1}$ reaches a step $n\ge 0$, then $F\subset C_n \cap Q_n$ and $x_{n+1}$ is well-defined.
\end{lemma}
\begin{proof} As mentioned above, the solution set $F$ is closed and convex. Further, by definitions, $C_n$ and $Q_n$ are the 
intersections of halfspaces with the closed convex subset $C$, hence they are closed and convex.\\
Next, we verify that $F\subset C_n\bigcap Q_n$ for all $n\ge 0$. For every $x^*\in F$, by the convexity of $||.||^2$, the nonexpansiveness 
of $S_j,$ and Lemma $\ref{lem.help1}$, we have
\begin{eqnarray}
||\bar{u}_n-x^*||^2&&=||\alpha_n x_n+(1-\alpha_n)S_{j_n} \bar{z}_n-x^*||^2 \nonumber\\
&&\le \alpha_n ||x_n-x^*||^2+(1-\alpha_n)||S_{j_n} \bar{z}_n-x^*||^2 \nonumber\\
&&\le \alpha_n||x_n-x^*||^2+(1-\alpha_n)||\bar{z}_n-x^*||^2\nonumber\\
&&\le \alpha_n||x_n-x^*||^2+(1-\alpha_n)||x_n-x^*||^2\nonumber\\
&&\le ||x_n-x^*||^2. \label{eq:2}
\end{eqnarray}
Therefore, $||\bar{u}_n-x^*||\le ||x_n-x^*||$ or $x^*\in C_n$. Hence $F\subset C_n$ for all $n\ge 0$.\\
 Now we show that $F\subset C_n\bigcap Q_n$ by induction. Indeed, we have $F\subset C_0$ as above. Besides, $F \subset C = Q_0,$ 
hence $F \subset C_0\bigcap Q_0.$   Assume that $F\subset C_{n-1}\bigcap Q_{n-1}$ for some $n\ge 1$. From 
$x_{n}=P_{C_{n-1}\bigcap Q_{n-1}}x_0$ and  $(\ref{eq:EquivalentPC})$, we get
$$
\left\langle x_{n}-z,x_0 -x_{n}\right\rangle \ge 0, \forall z\in C_{n-1} \bigcap Q_{n-1}.
$$
Since $F \subset C_{n-1} \bigcap Q_{n-1}$, $\left\langle x_{n}-z,x_0 -x_{n}\right\rangle \ge 0$ for all $z\in F$. This together with the 
definition of $Q_{n}$ implies that $F \subset Q_{n}$. Hence $F \subset C_{n} \bigcap Q_{n}$ for all $n\ge 1$. Since $F$ and $C_n \cap 
Q_n$ are nonempty closed convex subsets, $P_F x_0$ and $x_{n+1}:=P_{C_n \cap Q_n}(x_0)$ are well-defined.
\end{proof}
\begin{lemma}\label{lem:FiniteIteration}
If Algorithm $\ref{algorithm1}$ finishes at a finite iteration $n<\infty$, then $x_n$ is a common element of two sets $\cap_{i=1}^N 
EP(f_i)$ and $\cap_{j=1}^M F(S_j)$, i.e., $x_n\in F$.
\end{lemma}
\begin{proof}
If $x_{n+1}=x_n$ then $x_n = x_{n+1}=P_{C_n\cap Q_n}(x_0) \in C_n$. By the definition of $C_n$, $ ||\bar{u}_n-x_n||\le 
||x_n-x_n||=0,$ hence $\bar{u}_n=x_n$. From the definition of $j_n$, we obtain
$$
u_n^j=x_n, \forall j=1,\ldots, M.
$$
This together with the relations $u_n^j=\alpha_n x_n+(1-\alpha_n)S_j \bar{z}_n$ and $0<\alpha_n<1$ implies that $x_n=S_j \bar{z}_n$. 
Let $x^* \in F.$  By Lemma $\ref{lem.help1}$ and the nonexpansiveness of $S_j$, we get 
\begin{eqnarray*}
||x_n-x^*||^2&&= ||S_j \bar{z}_n-x^*||^2\\ 
&&\le||\bar{z}_n-x^*||^2\\
&&\le ||x_n-x^*||^2-(1-2\rho c_1)||y_n^{i_n}-x_n||^2-(1-2\rho c_2)||y_n^{i_n}-\bar{z}_n||^2.
\end{eqnarray*}
Therefore
$$ (1-2\rho c_1)||y_n^{i_n}-x_n||^2 + (1-2\rho c_2)||y_n^{i_n}-\bar{z}_n||^2\le 0. $$
Since $0<\rho < \min\left\{\frac{1}{2c_1},\frac{1}{2c_2}\right\}$, from the last inequality we obtain $x_n=y_n^{i_n}=\bar{z}_n$. 
Therefore $x_n=S_j \bar{z}_n=S_j x_n$ or $x_n \in F(S_j)$ for all $j=1,\ldots, M$. Moreover, from the relation $x_n=\bar{z}_n$ and 
the definition of $i_n$, we also get $x_n=z_n^i$ for all $i=1,\ldots, N$. This together with the inequality $(\ref{eq:1})$ implies that 
$x_n=y_n^{i}$ for all $i=1,\ldots,N$. Thus,
$$ x_n = {\rm argmin} \{ \rho f_i(x_n, y) +\frac{1}{2}||x_n-y||^2:  y \in C\}.$$
By  \cite[Proposition 2.1]{M2000}, from the last relation we conclude that   $x_n\in EP(f_i)$ for all $i=1,\ldots, N,$  hence $x_n \in F$. Lemma  $\ref{lem:FiniteIteration}$ is proved.
\end{proof}
\begin{lemma}\label{lem.limits}
Let $\left\{x_n\right\}, \left\{y_n^i\right\}, \left\{z_n^i\right\}, \left\{u_n^j\right\}$ be (infinite) sequences generated by Algorithm $\ref{algorithm1}$. Then, there hold the relations 
$$
\lim_{n\to\infty}||x_{n+1}-x_n||=\lim_{n\to\infty}||x_n-u_n^j||=\lim_{n\to\infty}||x_n-z_n^i||=\lim_{n\to\infty}||x_n-y_n^i||=0,
$$
and $\lim_{n\to\infty}||x_n-S_j x_n||=0.$
\end{lemma}
\begin{proof}
From the definition of $Q_n$ and $(\ref{eq:EquivalentPC})$, we see that  $x_n=P_{Q_n}x_0$. Therefore, for every $u\in 
F\subset Q_n$, we get
\begin{equation}\label{eq:3}
\left\|x_n-x_0 \right\|^2\le \left\|u-x_0 \right\|^2-\left\|u-x_n \right\|^2\le\left\|u-x_0 \right\|^2.
\end{equation}
This implies that the sequence $\left\{x_n\right\}$ is bounded. From $(\ref{eq:2})$, the sequence $\{\bar{u}_n\},$ and hence, the 
sequence $\left\{u_n^j\right\}$ are also bounded.\\  
Observing that $x_{n+1}=P_{C_n\bigcap Q_n}x_0\in Q_n, x_n=P_{Q_n}x_0$, from $(\ref{eq:ProperOfPC})$ we have 
\begin{equation}\label{eq:4}
\left\|x_n-x_0 \right\|^2\le \left\|x_{n+1}-x_0 \right\|^2-\left\|x_{n+1}-x_n \right\|^2\le \left\|x_{n+1}-x_0 \right\|^2.
\end{equation}
Thus, the sequence $\left\{\left\|x_n-x_0 \right\|\right\}$ is nondecreasing, hence there exists the limit of the sequence $\left\{\left\|x_n-x_0 
\right\|\right\}$. From $(\ref{eq:4})$ we obtain
$$
\left\|x_{n+1}-x_n \right\|^2\le \left\|x_{n+1}-x_0 \right\|^2-\left\|x_n-x_0 \right\|^2.
$$
Letting $n\to\infty$, we find
\begin{equation}\label{eq:5}
\lim_{n\to\infty}\left\|x_{n+1}-x_n \right\|=0.
\end{equation}
Since $x_{n+1}\in C_n$, $||\bar{u}_n-x_{n+1}||\le \left\|x_{n+1}-x_n \right\|$. Thus 
$||\bar{u}_n-x_n||\le||\bar{u}_n-x_{n+1}||+||x_{n+1}-x_n||\le 2||x_{n+1}-x_n||$. The last inequality together with $(\ref{eq:5})$ implies
that $||\bar{u}_n-x_n||\to 0$ as $n\to\infty$. From the definition of $j_n$, we conclude that
\begin{equation}\label{eq:6}
\lim_{n\to\infty}\left\|u_n^j-x_n \right\|=0
\end{equation}
 for all $j=1,\ldots, M$. Moreover, Lemma $\ref{lem.help1}$ shows that for any fixed $x^* \in F,$ we have
\begin{eqnarray*}
||u_n^j-x^*||^2&&=||\alpha_n x_n+(1-\alpha_n)S_j\bar{z}_n-x^*||^2\nonumber\\
&&\le \alpha_n ||x_n-x^*||^2+(1-\alpha_n)||S_j\bar{z}_n-x^*||^2\nonumber\\
&&\le \alpha_n ||x_n-x^*||^2+(1-\alpha_n)||\bar{z}_n-x^*||^2\nonumber\\
&&\le ||x_n-x^*||^2\nonumber\\
&&-(1-\alpha_n)||\left((1-2\rho c_1)||y_n^{i_n}-x_n||^2+(1-2\rho c_2)||y_n^{i_n}-\bar{z}_n||^2\right).\nonumber
\end{eqnarray*}
Therefore
\begin{eqnarray}
&&(1-\alpha_n)(1-2\rho c_1)||y_n^{i_n}-x_n||^2+(1-2\rho c_2)||y_n^{i_n}-\bar{z}_n||^2\nonumber\\
&&\le||x_n-x^*||^2-||u_n^{j}-x^*||^2\nonumber\\ 
&&= \left(||x_n-x^*||-||u_n^j-x^*||\right)\left(||x_n-x^*||+||u_n^j-x^*||\right)\nonumber\\
&&\le ||x_n-u_n^j||\left(||x_n-x^*||+||u_n^j-x^*||\right)\label{eq:6*2}.
\end{eqnarray}
Using the last inequality together with $(\ref{eq:6})$ and taking into account the boundedness of two sequences $\left\{u_n^j\right\}$, 
$\left\{x_n\right\}$ as well as the condition $\lim\sup_{n\to\infty}\alpha_n<1,$ we come to the relations
\begin{equation}\label{eq:8}
\lim_{n\to\infty}\left\|y_n^{i_n}-x_n \right\|=\lim_{n\to\infty}\left\|y_n^{i_n}-\bar{z}_n \right\|=0
\end{equation}
for all $i=1,\ldots, N$. From $||\bar{z}_n-x_n||\le||\bar{z}_n-y_n^{i_n}||+||y_n^{i_n}-x_n||$ and $(\ref{eq:8})$, we obtain 
$\lim_{n\to\infty}\left\|\bar{z}_n-x_n \right\|=0$. By the definition of $i_n$, we get
\begin{equation}\label{eq:8*}
\lim_{n\to\infty}\left\|z_n^i-x_n \right\|=0
\end{equation}
for all $i=1,\ldots,N$. From Lemma $\ref{lem.help1}$ and $(\ref{eq:8*})$, arguing similarly to $(\ref{eq:6*2})$ we obtain
\begin{equation}\label{eq:8*1}
\lim_{n\to\infty}\left\|y_n^i-x_n \right\|=0
\end{equation}
for all $i=1,\ldots,N$.
On the other hand, since $u_n^j=\alpha_n x_n+(1-\alpha_n)S_j \bar{z}_n$, we have
\begin{eqnarray*}
||u_n^j-x_n||&&=(1-\alpha_n)||S_j \bar{z}_n-x_n||\\ 
&&=(1-\alpha_n)||(S_j x_n-x_n)+(S_j \bar{z}_n-S_jx_n)||\\ 
&&\ge (1-\alpha_n)\left(||S_j x_n-x_n||-||S_j \bar{z}_n-S_jx_n||\right)\\ 
&&\ge (1-\alpha_n)\left(||S_j x_n-x_n||-||\bar{z}_n-x_n||\right).
\end{eqnarray*}
Therefore 
$$||S_j x_n-x_n||\le||\bar{z}_n-x_n||+\frac{1}{1-\alpha_n}||u_n^j-x_n||.$$
The last inequality together with $(\ref{eq:6}),(\ref{eq:8*})$ and the condition \\
$\lim\sup_{n\to\infty}\alpha_n<1$ implies that
\begin{equation}\label{eq:9}
\lim_{n\to\infty}\left\|S_j x_n-x_n \right\|=0,
\end{equation}
for all $j=1,\ldots,M$. The proof of Lemma $\ref{lem.limits}$ is complete.
\end{proof}
\begin{lemma}\label{lem.weakly_limit_point}
Let $\left\{x_n\right\}$ be the sequence generated by Algorithm $\ref{algorithm1}$. Suppose that $\bar{x}$ is a weak limit point of $\left\{x_n\right\}$. Then $\bar{x}\in F=\left(\bigcap_{i=1}^N EP(f_i)\right)\bigcap \left(\bigcap_{j=1}^M F(S_j)\right)$, i.e., $\bar{x}$ is a common element of the set of solutions of equilibrium problems for bifunctions $\left\{f_i\right\}_{i=1}^N$ and the set of fixed points of nonexpansive mappings $\left\{S_j\right\}_{j=1}^M$.
\end{lemma}
\begin{proof}
From Lemma $\ref{lem.limits}$ we see that $\left\{x_n\right\}$ is bounded. 
Then there exists a subsequence of $\left\{x_n\right\}$ converging weakly to $\bar{x}$. 
For the sake of simplicity, we denote the weakly convergent subsequence again by $\left\{x_n\right\},$ i.e.,  $x_n\rightharpoonup \bar{x}$. 
From $(\ref{eq:9})$ and the demiclosedness of $I - S_j$, we have $\bar{x}\in F(S_j)$. Hence, $\bar{x}\in \bigcap_{j=1}^M F(S_j)$. 
Noting that
$$ y_n^i = {\rm argmin} \{ \rho f_i(x_n, y) +\frac{1}{2}||x_n-y||^2:  y \in C\}, $$
by Lemma $\ref{lem.Equivalent_MinPro}$, we obtain
$$
0\in \partial_2 \left\{\rho f_i(x_n,y)+\frac{1}{2}||x_n-y||^2\right\}(y_n^i)+N_C(y_n^i).
$$
Therefore, there exist $w\in \partial_2 f_i (x_n,y_n^i)$ and $\bar{w}\in N_C(y_n^i)$ such that
\begin{equation}\label{eq:10}
\rho w+x_n-y_n^i+\bar{w}=0.
\end{equation}
Since $\bar{w}\in N_C(y_n^i)$, $\left\langle \bar{w}, y-y_n^i\right\rangle \le 0$ for all $y\in C$. This together with $(\ref{eq:10})$ implies 
that
\begin{equation}\label{eq:11}
\rho \left\langle w, y-y_n^i\right\rangle \ge \left\langle y_n^i-x_n, y-y_n^i\right\rangle
\end{equation}
for all $y\in C$. Since $w\in \partial_2 f_i (x_n,y_n^i)$, 
\begin{equation}\label{eq:12}
f_i(x_n,y)-f_i(x_n,y_n^i)\ge \left\langle w, y-y_n^i\right\rangle, \forall y\in C.
\end{equation}
From $(\ref{eq:11})$ and $(\ref{eq:12})$, we get
\begin{equation}\label{eq:13}
\rho \left(f_i(x_n,y)-f_i(x_n,y_n^i)\right) \ge \left\langle y_n^i-x_n, y-y_n^i\right\rangle, \forall y\in C.
\end{equation}
Since  $x_n\rightharpoonup \bar{x}$ and $||x_n-y_n^i||\to 0$ as $n\to\infty,$ we find   $ \quad y_n^i\rightharpoonup \bar{x}$. 
Letting $n\to\infty$ in $(\ref{eq:13})$ and using assumption A3, we conclude that $f_i(\bar{x},y)\ge 0$ for all
 $y \in C$ (i=1,\ldots,N). Thus, $\bar{x}\in \bigcap_{i=1}^N EP(f_i),$ hence $\bar{x}\in F$. The proof of Lemma $\ref{lem.weakly_limit_point}$ is complete.
\end{proof}
\begin{theorem}\label{theorem1}
Let C be a nonempty closed convex subset of a real Hilbert space $H$. Suppose that $\left\{f_i\right\}_{i=1}^N$ is a finite family of 
bifunctions satisfying conditions ${\rm A1-A4}$ and $\left\{S_j\right\}_{j=1}^M$ is a finite family of nonexpansive mappings on $C.$
   Moreover, suppose that the solution set $F$ is nonempty. Then, the (infinite) sequence $\left\{x_n\right\}$ generated by Algorithm 
$\ref{algorithm1}$ converges strongly to $x^\dagger=P_F x_0$.
\end{theorem}
\begin{proof} It is directly followed from Lemma $\ref{lem:CQ_closedconvex}$ that the sets $F, C_n, Q_n$ are closed convex subsets of $C$ and 
$F\subset C_n\bigcap Q_n$ for all $n\ge 0$. Moreover, from Lemma $\ref{lem.limits}$ we see that the sequence $\left\{x_n\right\}$ is bounded. Suppose that $\bar{x}$ is any weak limit point of $\left\{x_n\right\}$ and $x_{n_j}\rightharpoonup \bar{x}$. By Lemma $\ref{lem.weakly_limit_point}$, $\bar{x}\in F$. We now show that the sequence $\left\{x_n\right\}$ converges strongly to $x^\dagger:=P_F x_0$. Indeed, from $x^\dagger \in F$ and $(\ref{eq:3})$, we obtain
$$
||x_{n_j}-x_0||\le||x^\dagger-x_0||.
$$
The last inequality together with $x_{n_j}\rightharpoonup \bar{x}$ and the weak lower semicontinuity of the norm $||.||$ implies that 
$$
||\bar{x}-x_0||\le \lim\inf_{j\to\infty}||x_{n_j}-x_0||\le \lim\sup_{j\to\infty}||x_{n_j}-x_0||\le||x^\dagger-x_0||.
$$
By the definition of $x^\dagger$, $\bar{x}=x^\dagger$ and $\lim_{j\to\infty}||x_{n_j}-x_0||=||x^\dagger-x_0||$. 
Since $x_{n_j} - x_0 \rightharpoonup \bar{x}- x_0 = x^\dagger - x_0,$  the Kadec-Klee property of the Hilbert space $H$ ensures that 
$x_{n_j} - x_0 \to x^\dagger - x_0,$ hence $x_{n_j}  \to x^\dagger$ as $j \to \infty.$ Since $\bar{x} =x^\dagger$ is any weak limit point of $\left\{x_n\right\}$, the sequence $\left\{x_n\right\}$ converges strongly to $x^\dagger:=P_F x_0$. The proof of Theorem $\ref{theorem1}$ is complete.
\end{proof}
\begin{corollary}\label{cor1}
Let C be a nonempty closed convex subset of a real Hilbert space $H$. Suppose that $\left\{f_i\right\}_{i=1}^N$ is a finite family of 
bifunctions satisfying conditions $A1-A4$, and the set $F=\bigcap_{i=1}^N EP(f_i)$ is nonempty. Let $\left\{x_n\right\}$ be the 
sequence generated in the following manner:
$$
\left\{
\begin{array}{ll}
&x_0\in C_0:=C, Q_0:=C,\\
&y_n^i = {\rm argmin} \{ \rho f_i(x_n, y) +\frac{1}{2}||x_n-y||^2:  y \in C\} \quad i=1,\ldots,N,\\
&z_n^i = {\rm argmin} \{ \rho f_i(y_n^i, y) +\frac{1}{2}||x_n-y||^2:  y \in C\} \quad i=1,\ldots,N,\\
&i_n = {\rm argmax}\{||z_n^i - x_n||: i =1,\ldots,N\},\bar{z}_n:=z^{i_n}_n,\\
&C_n = \{v \in C: ||\bar{z}_n - v|| \leq ||x_n-v||\},\\
&Q_n = \{ v \in C:  \langle x_0 - x_n, v - x_n \rangle \leq 0 \},\\
&x_{n+1}=P_{C_n\bigcap Q_n}x_0,n\ge 0,
\end{array}
\right.
$$
where $ 0< \rho <\min \left(\frac{1}{2c_1},\frac{1}{2c_2}\right)$. Then the sequence $\left\{x_n\right\}$ converges strongly to 
$x^\dagger=P_F x_0$.
\end{corollary}
\begin{corollary}\label{cor2}
Let C be a nonempty closed convex subset of a real Hilbert space $H$. Suppose that $\left\{A_i\right\}_{i=1}^N$ is a finite family of 
pseudomonotone and $L$-Lipschitz continuous mappings from $C$ to $H$ such that $F=\bigcap_{i=1}^N VI(A_i,C)$ is nonempty, 
where $VI(A_i,C)=\left\{x^*\in C: \left\langle A(x^*),y-x^*\right\rangle\ge 0,~\forall y\in C\right\}$. 
Let $\left\{x_n\right\}$ be the sequence generated in the following manner:
$$
\left\{
\begin{array}{ll}
&x_0\in C_0:=C, Q_0:=C,\\
&y_n^i = P_C\left(x_n-\rho A_i(x_n)\right) \quad i=1,\ldots,N,\\
&z_n^i = P_C\left(x_n-\rho A_i(y_n^i)\right) \quad i=1,\ldots,N,\\
&i_n = {\rm argmax}\{||z_n^i - x_n||: i =1,\ldots,N\},\bar{z}_n:=z^{i_n}_n,\\
&C_n = \{v \in C: ||\bar{z}_n - v|| \leq ||x_n-v||\},\\
&Q_n = \{ v \in C:  \langle x_0 - x_n, v - x_n \rangle \leq 0 \},\\
&x_{n+1}=P_{C_n\bigcap Q_n}x_0,n\ge 0,
\end{array}
\right.
$$
where $ 0< \rho <\frac{1}{L}$. Then the sequence $\left\{x_n\right\}$ converges strongly to $x^\dagger=P_F x_0$.
\end{corollary}
\begin{proof}
Let $f_i(x,y)=\left\langle A_i(x),y-x\right\rangle$ for all $x,y\in C$ and $i=1,\ldots,N$.\\
 Since $A_i$ is $L$-Lipschitz continuous, for all 
$x,y,z\in C$
\begin{eqnarray*}
f_i(x,y)+f_i(y,z)-f_i(x,z)&&=\left\langle A_i(x),y-x\right\rangle+\left\langle A_i(y),z-y\right\rangle-\left\langle A_i(x),z-x\right\rangle\\ 
&&=-\left\langle A_i(y)-A_i(x),y-z\right\rangle\\
&&\ge -||A_i(y)-A_i(x)|||y-z||\\
&&\ge -L||y-x||||y-z||\\
&&\ge -\frac{L}{2}||y-x||^2-\frac{L}{2}||y-z||^2.
\end{eqnarray*}
Therefore $f_i$ is Lipschitz-type continuous with $c_1=c_2=\frac{L}{2}$. Moreover, the pseudomonotonicity of $A_i$ ensures the 
pseudomonotonicity of  $f_i$. Conditions A3, A4 are satisfied automatically. According to Algorithm $\ref{algorithm1}$, we have
\begin{eqnarray*}
&&y_n^i = {\rm argmin} \{ \rho\left\langle A_i(x_n),y-x_n\right\rangle +\frac{1}{2}||x_n-y||^2:  y \in C\},\\ 
&&z_n^i = {\rm argmin} \{ \rho\left\langle A_i(y_n^i),y-y_n^i\right\rangle +\frac{1}{2}||x_n-y||^2:  y \in C\}.
\end{eqnarray*}
Or
\begin{eqnarray*}
&&y_n^i = {\rm argmin} \{\frac{1}{2}||y-(x_n-\rho A_i(x_n) )||^2:  y \in C\}=P_C(x_n-\rho A_i(x_n) ),\\ 
&&z_n^i = {\rm argmin} \{ \frac{1}{2}||y-(x_n-\rho A_i(y^i_n) )||^2:  y \in C\}=P_C(x_n-\rho A_i(y^i_n)).
\end{eqnarray*}
Application to Theorem $\ref{theorem1}$ with the above mentioned $f_i(x,y),(i=1,\ldots,N)$ and $S_j=I, (j=1,\ldots,M)$ leads to the 
desired result.
\end{proof}
\begin{remark}
Putting $N=1$ in Corollary $\ref{cor2}$, we obtain the corresponding result of Nadezhkina and 
Takahashi  \cite[Theorem 4.1]{NT2006}.
\end{remark}
\noindent Now, replacing Mann's iteration in Step 4 of Algorithm $\ref{algorithm1}$ by Halpern's one, we come to the following
 algorithm.
\begin{algorithm}\label{algorithm1add} {\rm (Parallel hybrid Halpern-extragradient method)}\\
\noindent \textbf{Initialization} $x^0 \in C, 0< \rho <\min \left(\frac{1}{2c_1},\frac{1}{2c_2}\right),  n := 0$ and the sequence 
$\left\{\alpha_k\right\} \subset (0,1)$ satisfies the condition $\lim_{k\to\infty}\alpha_k=0$.\\
\textbf{Step 1.} Solve $N$ strongly convex programs in parallel
$$y_n^i = {\rm argmin} \{ \rho f_i(x_n, y) +\frac{1}{2}||x_n-y||^2:  y \in C\} \quad i=1,\ldots,N.$$
\textbf{Step 2.} Solve $N$ strongly convex programs in parallel
 $$ z_n^i = {\rm argmin} \{ \rho f_i(y_n^i, y) +\frac{1}{2}||x_n-y||^2:  y \in C\} \quad i=1,\ldots,N. $$
\textbf{Step 3.} Find among $z_n^i, \quad i =1,\ldots,N,$ the farthest element from $x_n$, i.e., 
$$ i_n = {\rm argmax}\{||z_n^i - x_n||: i =1,\ldots,N\},\bar{z}_n:=z^{i_n}_n. $$
\textbf{Step 4.} Find intermediate approximations $u_n^j$ in parallel
$$ u_n^j=\alpha_n x_0+(1-\alpha_n)S_j \bar{z}_n, j=1,\ldots,M. $$
\textbf{Step 5.} Find among $u_n^j, \quad j =1,\ldots,M,$ the farthest element from $x_n$, i.e., 
$$ j_n= {\rm argmax}\{||u_n^j - x_n||: j =1,\ldots,M\},\bar{u}_n:=u^{j_n}_n. $$
\textbf{Step 6.} Construct two closed convex subsets of C
\begin{eqnarray*}
&&C_n = \{v \in C: ||\bar{u}_n - v||^2\leq \alpha_n||x_0-v||^2+(1-\alpha_n)||x_n-v||^2\},\\ 
&&  Q_n = \{ v \in C:  \langle x_0 - x_n, v - x_n \rangle \leq 0 \}.
\end{eqnarray*}
\textbf{Step 7.} The next approximation $x_{n+1}$ is defined as the projection of $x_0$ onto $C_n \cap Q_n$, i.e.,
$$ x_{n+1}= P_{C_n \cap Q_n}(x_0). $$
\textbf{Step 8.} Put  $n:=n+1$ and go to \textbf{Step 1}.
\end{algorithm}
\begin{remark}
For Algorithm 2, the claim that $x_n$ is a common solution of the equlibrium and fixed point problems, if $x_{n+1} = x_n,$  in 
general is not true. So in practice, we need to use some "stopping rule" like if $n  > n_{\max}$  for some chosen sufficiently large number 
$n_{\max},$  then stop.
\end{remark}
\begin{theorem}\label{theorem1add}
Let C be a nonempty closed convex subset of a real Hilbert space $H$. Suppose that $\left\{f_i\right\}_{i=1}^N$ is a finite family of 
bifunctions satisfying conditions $A1-A4,$ and $\left\{S_j\right\}_{j=1}^M$ is a finite family of nonexpansive mappings on $C.$  
Moreover, suppose that the solution set $F$ is nonempty. Then, the sequence $\left\{x_n\right\}$ generated by the Algorithm 
$\ref{algorithm1add}$ converges strongly to $x^\dagger=P_F x_0$.
\end{theorem}
\begin{proof}
Arguing similarly as in the proof of Lemma $\ref{lem:CQ_closedconvex}$ and Theorem $\ref{theorem1}$, we conclude that $F, C_n, Q_n$ 
are closed and convex. Besides, $F\subset C_n\cap Q_n$ for all $n\ge 0$. Moreover, the sequence $\left\{x_n\right\}$ is bounded and 
\begin{equation}\label{eq:1add}
\lim_{n\to\infty}||x_{n+1}-x_n||=0.
\end{equation}
Since $x_{n+1}\in C_{n+1}$,
$$
||\bar{u}_n-x_{n+1}||^2\le \alpha_n||x_0-x_{n+1}||^2+(1-\alpha_n)||x_n-x_{n+1}||^2.
$$
Letting $n\to\infty$, from $(\ref{eq:1add})$, $\lim_{n\to\infty}\alpha_n=0$ and the boundedness of $\left\{x_n\right\}$, we obtain
$$
\lim_{n\to\infty}||\bar{u}_n-x_{n+1}||=0.
$$
Proving similarly to $(\ref{eq:6})$ and $(\ref{eq:6*2})$, we get
$$
\lim_{n\to\infty}||u_n^j-x_{n}||=0, \quad j=1,\ldots,M,
$$
and
\begin{eqnarray}
&&(1-\alpha_n)(1-2\rho c_1)||y_n^{i_n}-x_n||^2+(1-2\rho c_2)||y_n^{i_n}-\bar{z}_n||^2\nonumber\\
&&\le\alpha_n(||x_0-x^*||^2-||x_n-x^*||^2)\nonumber\\ 
&&+||x_n-u_n^j||\left(||x_n-x^*||+||u_n^j-x^*||\right) \label{eq:5add}
\end{eqnarray}
for each $x^*\in F$. Letting $n\to\infty$ in $(\ref{eq:5add})$, one has
$$
\lim_{n\to\infty}||y_n^{i_n}-x_{n}||=\lim_{n\to\infty}||\bar{z}_n-x_{n}||=0, \quad j=1,\ldots,N,
$$
Repeating the proof of $(\ref{eq:8*})$ and $(\ref{eq:8*1})$, we get
$$
\lim_{n\to\infty}||y_n^i-x_{n}||=\lim_{n\to\infty}||z_n^i-x_{n}||=0, \quad i=1,\ldots,N.
$$
Using $u_n^j=\alpha_n x_0+(1-\alpha_n)S_j \bar{z}_n$, by a straightforward computation, we obtain
$$
||S_j x_n-x_n||\le||\bar{z}_n-x_n||+\frac{1}{1-\alpha_n}||u_n^j-x_n||+\frac{\alpha_n}{1-\alpha_n}||x_0-x_n||,
$$
which implies that $\lim_{n\to\infty}||S_j x_n-x_n||=0$. The rest of the proof of Theorem $\ref{theorem1add}$ is similar to the arguments 
in the proofs of Lemma $\ref{lem.weakly_limit_point}$ and Theorem $\ref{theorem1}$.
\end{proof}
\noindent Next replacing Steps 4 and 5 in Algorithm 1, consisting of a Mann's iteration and a parallel splitting-up step, by an iteration step 
involving a convex combination of the identity mapping $I$ and the mappings $S_j, ~ j =1,\ldots, N$, we come to the following algorithm.
\begin{algorithm}\label{algorithm2} {\rm (Parallel hybrid iteration-extragradient method)}\\
\noindent \textbf{Initialization}: $x^0 \in C, 0< \rho <\min \left(\frac{1}{2c_1},\frac{1}{2c_2}\right) ,  n := 0$ and the positive sequences 
$\left\{\alpha_{k,l}\right\}_{k=1}^\infty(l=0,\ldots,M)$ satisfy the conditions: $0\le\alpha_{k,j}\le 1$, $\sum_{j=0}^M \alpha_{k,j}=1$, 
$\lim\inf_{k\to\infty}\alpha_{k,0}\alpha_{k,l}>0$ for all $l=1,\ldots,M$.\\
\textbf{Step 1.} Solve $N$ strongly convex programs in parallel
$$y_n^i = {\rm argmin} \{ \rho f_i(x_n, y) +\frac{1}{2}||x_n-y||^2:  y \in C\} \quad i=1,\ldots,N.$$
\textbf{Step 2.} Solve $N$ strongly convex programs in parallel
 $$ z_n^i = {\rm argmin} \{ \rho f_i(y_n^i, y) +\frac{1}{2}||x_n-y||^2:  y \in C\} \quad i=1,\ldots,N. $$
\textbf{Step 3.} Find among $z_n^i, \quad i =1,\ldots,N,$ the farthest element from $x_n$, i.e., 
$$ i_n = {\rm argmax}\{||z_n^i - x_n||: i =1,\ldots,N\},\bar{z}_n:=z^{i_n}_n. $$
\textbf{Step 4.} Compute in parallel  $u_n^j := S_j\bar{z}_n; ~ j=1,\ldots, M,$  and put
$$ u_n=\alpha_{n,0}x_n+\sum_{j=1}^M \alpha_{n,j} u_n^j.$$\\
\textbf{Step 5.} Construct two closed convex subsets of C
\begin{eqnarray*}
&&C_n = \{v \in C: ||u_n - v||\leq ||x_n-v||\},\\ 
&& Q_n = \{ v \in C:  \langle x_0 - x_n, v - x_n \rangle \leq 0 \}.
\end{eqnarray*}
\textbf{Step 6.} The next approximation $x_{n+1}$ is determined as the projection of $x_0$ onto $C_n \cap Q_n$, i.e.,
$$ x_{n+1}= P_{C_n \cap Q_n}(x_0). $$
\textbf{Step 7.} If  $x_{n+1}=x_n$ then stop. Otherwise, set $n:= n+1$ and go to \textbf{Step 1}.
\end{algorithm}
\begin{remark}
Arguing similarly as in the proof of Lemma $\ref{lem:FiniteIteration}$, we can prove that if Algorithm $\ref{algorithm2}$ finishes at a 
finite iteration $n<\infty,$ then $x_n\in F$, i.e., $x_n$ is a common element of the set of solutions of  equilibrium problems and the set 
of fixed points of  nonexpansive mappings.
\end{remark}
\begin{theorem}\label{theorem2}
Let C be a nonempty closed convex subset of a real Hilbert space $H$. Suppose that $\left\{f_i\right\}_{i=1}^N$ is a finite family of 
bifunctions satisfying conditions $A1-A4,$ and $\left\{S_j\right\}_{j=1}^M$ is a finite family of nonexpansive mappings on $C.$  
Moreover, suppose that the solution set $F$ is nonempty. Then, the (infinite) sequence $\left\{x_n\right\}$ generated by the Algorithm 
$\ref{algorithm2}$ converges strongly to $x^\dagger=P_F x_0$.
\end{theorem}
\begin{proof}
Arguing similarly as in the proof of Theorem $\ref{theorem1}$, we can conclude that $F, C_n, Q_n$ are closed convex subsets of $C$. 
Besides, $F\subset C_n \bigcap Q_n$ and
\begin{equation}\label{eq:14}
\lim_{n\to\infty}||x_{n+1}-x_n||=\lim_{n\to\infty}||y_{n}^i-x_n||=\lim_{n\to\infty}||z_{n}^i-x_n||=\lim_{n\to\infty}||u_n-x_n||=0
\end{equation}
for all $i=1,\ldots, N$. For every $x^*\in F$, by Lemmas $\ref{lem:PU2008}$ and $\ref{lem.help1}$, we have
\begin{eqnarray*}
||u_n-x^*||^2&&=||\alpha_{n,0}x_n+\sum_{j=1}^M \alpha_{n,j}S_j\bar{z}_n-x^*||^2\\ 
&& =||\alpha_{n,0}(x_n-x^*)+\sum_{j=1}^M\alpha_{n,j} (S_j\bar{z}_n-x^*)||^2\\
&&\le \alpha_{n,0}||x_n-x^*||^2+\sum_{j=1}^M \alpha_{n,j}||S_j\bar{z}_n-x^*||^2-\alpha_{n,0}\alpha_{n,l}g(||S_l\bar{z}_n-x_n||)\\
&&\le \alpha_{n,0}||x_n-x^*||^2+\sum_{j=1}^M \alpha_{n,j}||\bar{z}_n-x^*||^2-\alpha_{n,0}\alpha_{n,l}g(||S_l\bar{z}_n-x_n||)\\
&&\le \alpha_{n,0}||x_n-x^*||^2+\sum_{j=1}^M \alpha_{n,j}||x_n-x^*||^2-\alpha_{n,0}\alpha_{n,l}g(||S_l\bar{z}_n-x_n||)\\
&&\le ||x_n-x^*||^2-\alpha_{n,0}\alpha_{n,l}g(||S_l\bar{z}_n-x_n||).
\end{eqnarray*}
Therefore 
\begin{eqnarray*}
\alpha_{n,0}\alpha_{n,l}g(||S_l\bar{z}_n-x_n||)&&\le||x_n-x^*||^2-||u_n-x^*||^2\\ 
&& \le\left(||x_n-x^*||-||u_n-x^*||\right)\left(||x_n-x^*||+||u_n-x^*||\right)\\
&& \le||x_n-u_n||\left(||x_n-x^*||+||u_n-x^*||\right).
\end{eqnarray*}
The last inequality together with $(\ref{eq:14})$, $\lim\inf_{n\to\infty}\alpha_{n,0}\alpha_{n,l}>0$ and the boundedness of 
$\left\{x_n\right\},\left\{u_n\right\}$ implies that
$\lim_{n\to\infty}g(||S_l\bar{z}_n-x_n||)=0$. Hence 
\begin{equation}\label{eq:15}
\lim_{n\to\infty}||S_l\bar{z}_n-x_n||=0.
\end{equation}
Moreover, from $(\ref{eq:14}),(\ref{eq:15})$ and $||S_l x_n-x_n||\le||S_l x_n-S_l 
\bar{z}_n||+||S_l\bar{z}_n-x_n||\le||x_n-\bar{z}_n||+||S_l\bar{z}_n-x_n||$ we obtain
$$
\lim_{n\to\infty}||S_lx_n-x_n||=0
$$
for all $l=1,\ldots,M$. The same argument as in the proofs of Lemma $\ref{lem.weakly_limit_point}$ and Theorem $\ref{theorem1}$ shows that the sequence 
$\left\{x_n\right\}$ converges strongly to $x^\dagger:=P_Fx_0$. The proof of Theorem $\ref{theorem2}$ is complete.
\end{proof}
\begin{remark}
Putting  $M=N=1$ in Theorems $\ref{theorem1}$ and $\ref{theorem2}$, we obtain the corresponding result announced in \cite[Theorem 
3.1]{A2011}.
\end{remark}
\section{Numerical experiment}
Let $H = \Re^1$ be a Hilbert space  with the standart inner product $\left\langle x,y \right\rangle := xy$ and the norm $||x|| := |x| $ 
for all $x, y \in H$.  Consider the bifunctions defined on the set $C:= [0, 1] \subset H$  by 
$$f_i(x,y) := B_i(x)(y-x),  i = 1,\ldots, N,$$ 
where $ B_i(x) = 0$ if  $0 \leq x \leq \xi_i,$ and $B_i(x) = \exp(x-\xi_i)+\sin(x-\xi_i)-1$ if $\xi_i \leq x \leq 1.$ Here $0<\xi_1 < \ldots < 
\xi_N< 1$. Obviously, conditions A3, A4 for the bifunctions $f_i$ are satisfied. Further, since $B_i(x)$ is nondecreasing on $[0,1]$,
$$ f_i(x,y)+f_i(y,x)=(x-y)(B_i(y)-B_i(x))\le 0. $$
Thus, each bifunction $f_i$ is monotone, and so is pseudomonotone. Moreover, $B_i(x)$ is 4-Lipschitz continuous. A 
straightforward calculation yields  $f_i(x,y) + f_i(y,z) - f_i(x,z) = (y-z)(B_i(x)-B_i(y)) \ge -4|x-y||y-z| \ge -2(x-y)^2 - 2(y-z)^2,$
which proves the Lipschitz-type continuity of $f_i$ with $c_1=c_2=2$. Finally,
$$ f_i(x,y)=B_i(x)(y-x)\ge 0,\quad \forall y\in [0,1] $$
if and only if $0\le x\le \xi_i$, i.e., $EP(f_i)=[0,\xi_i]$. Therefore $\cap_{i=1}^N EP(f_i)=[0,\xi_1].$\\
Define the mappings 
$$S_jx:= \frac{x^j\sin^{j-1}(x)}{2j-1}, \quad j=1,\ldots, M.$$
Clearly, $S_j: C \to C$ and 
$$|{S_j}'(x)| =\frac{1}{2j-1}|jx^{j-1}\sin^{j-1}(x)+(j-1)x^j\sin^{j-2}(x)\cos(x)| \leq 1.$$
Hence $S_j ,~ j=1,\ldots,M$ are nonexpansive mappings. Moreover, $F(S_1)=[0,1]$ and $F(S_j)=\left\{0\right\},~ j = 2, \ldots, M.$
Thus, the solution set $$F=\left(\cap_{i=1}^N EP(f_i)\right)\bigcap \left(\cap_{j=1}^M F(S_j)\right) = \{ 0\}.$$
By Algorithm $\ref{algorithm1}$, we have
\begin{equation}\label{eq:ex.16}
y_n^i=\arg\min \left\{\rho B_i(x_n)(y-x_n)+\frac{1}{2}(y-x_n)^2:y\in [0;1]\right\}.
\end{equation}
A simple computation shows that $(\ref{eq:ex.16})$  is equivalent to the following relation
$$
y_n^i=x_n-\rho B_i(x_n),\quad i=1,\ldots,N.
$$
Similarly, we obtain
\begin{equation}\label{eq:ex.18}
z_n^i=x_n-\rho B_i(y_n^i),\quad i=1,\ldots,N.
\end{equation}
From $(\ref{eq:ex.18})$, we can find the itermediate approximation $\bar{z}_n$ which is the farthest from $x_n$ among $z_n^i,\, i=1,\ldots,N.$
 Therefore,
\begin{equation}\label{eq:ex.19}
u_n^j=\alpha_n x_n+(1-\alpha_n)\frac{\bar{z}_n^j\sin^{j-1}(\bar{z}_n)}{2j-1}, \,j=1,\ldots,M.
\end{equation}
From $(\ref{eq:ex.19})$, we can find the intermediate approximation $\bar{u}_n$ which is farthest from $x_n$ among $u_n^j,\, 
j=1,\ldots,M$. By Lemma $\ref{lem:FiniteIteration}$, if $x_n=\bar{u}_n$, $x_n=0 \in F$. Otherwise, if $x_n>\bar{u}_n\ge 0$, by the proof of Theorem $\ref{theorem1}$, 
$0\in C_n$, i.e., $|\bar{u}_n|\le |x_n|,$ hence $0\le\bar{u}_n< x_n$. This together with the definitions of $C_n$ and $Q_n$ lead us to 
the following formulas:
\begin{eqnarray*}
&&C_n=\left[0,\frac{x_n+\bar{u}_n}{2}\right];\\
&&Q_n=[0,x_n].
\end{eqnarray*}
Therefore 
$$
C_n\cap Q_n=\left[0,\min \left\{x_n,\frac{x_n+\bar{u}_n}{2}\right\}\right].
$$
Since $\bar{u}_n\le x_n$, we find $\frac{x_n+\bar{u}_n}{2}\le x_n$. So
$$
C_n\cap Q_n=\left[0,\frac{x_n+\bar{u}_n}{2}\right].
$$
From the definition of $x_{n+1}$ we obtain
$$
x_{n+1}=\frac{x_n+\bar{u}_n}{2}.
$$
Thus we come to the following algorithm:\\
\textbf{Initialization} $x_0:=1$;  $n:=1$;  $\rho:=1/5$;  $\alpha_n:=1/n$;  $\epsilon:=10^{-5}$;  $\xi_i:=i/(N+1)$, $i=1,\ldots,N$; 
$N:=2\times 10^6$;  $M:=3\times 10^6$.\\
\textbf{Step 1.} Find the intermediate approximations $y_n^i$ in parallel $(i=1,\ldots,N)$.
$$
y_n^i=
\left\{
\begin{array}{ll}
&x_n \quad \mbox{if}\quad 0\le x_n \le \xi_i,\\
& x_n-\rho[\exp(x_n-\xi_i)+\sin(x_n-\xi_i)-1]  \quad \mbox{if}\quad \xi_i <x_n \le 1.
\end{array}
\right.
$$
\textbf{Step 2.} Find the intermediate approximations $z_n^i$ in parallel $(i=1,\ldots,N)$.
$$
z_n^i=
\left\{
\begin{array}{ll}
&x_n \quad \mbox{if}\quad 0\le y_n^i \le \xi_i,\\
& x_n-\rho[\exp(y_n^i-\xi_i)+\sin(y_n^i-\xi_i)-1]  \quad \mbox{if}\quad \xi_i <y_n^i \le 1.
\end{array}
\right.
$$
\textbf{Step 3.} Find the element $\bar{z}_n$ which is farthest from $x_n$ among $z_n^i,i=1,\ldots,N$.
$$
i_n=\arg\max\left\{|z_n^i-x_n|:i=1,\ldots,N\right\},\, \bar{z}_n=z_n^{i_n}.
$$
\textbf{Step 4.} Find the intermediate approximations $u_n^j$ in parallel
$$
u_n^j=\alpha_n x_n+(1-\alpha_n)\frac{\bar{z}_n^j\sin^{j-1}(\bar{z}_n)}{2j-1}, \,j=1,\ldots,M.
$$
\textbf{Step 5.} Find the element $\bar{u}_n$ which is farthest from $x_n$ among $u_n^j,j=1,\ldots,M$.
$$
j_n=\arg\max\left\{|u_n^j-x_n|:j=1,\ldots,M\right\},\, \bar{u}_n=z_n^{j_n}.
$$
\textbf{Step 6.} If $|\bar{u}_n-x_n|\le \epsilon$ then stop. Otherwise go to \textbf{Step 7.}\\
\textbf{Step 7.} $x_{n+1}=\frac{x_n+\bar{u}_n}{2}$.\\
\textbf{Step 8.} If $|x_{n+1}-x_n|\le \epsilon $ then stop. Otherwise, set $n:=n+1$ and go to \textbf{Step 1.}\\\\

\noindent The numerical experiment is performed on a LINUX cluster 1350 with 8 computing nodes. Each node contains two Intel Xeon 
dual core 3.2
 GHz, 2GBRam. All the programs are written in C.\\
For given tolerances we compare execution time of  the parallel hybrid Mann-extragradient method (PHMEM) in parallel 
and sequential modes.\\ 
We use the following notations:
{\begin{center}
\begin{tabular}{l l}
PHMEM & The parallel hybrid Mann-extragradient method\\
$TOL$ & Tolerance $\|x_k - x^*\|$ \\
$T_p$  & Time for PHMEM's execution in parallel mode (2CPUs - in seconds)\\
$T_s$ & Time for PHMEM's execution in sequential mode (in seconds) \\
\end{tabular}
\end{center}
\begin{table}[ht]\caption{Experiment with $\alpha_n = \frac{1}{n}. $}\label{tab:1}
\medskip\begin{center}
\begin{tabular}{|c|c|c|}
\hline
 $\qquad TOL \qquad$ & \multicolumn{2}{c|}{PHMEM}
\\ \cline{2-3}
& $\qquad T_p\qquad$ & $\qquad T_s\qquad$ \\ \hline
 $10^{-5}$ & 5.23 & 9.98 \\
 $10^{-6}$ & 5.86 & 11.25  \\
  $10^{-8}$ & 7.57 & 14.33  \\ \hline
 \end{tabular}
\end{center}
\end{table}
According to the above experiment,  in the most favourable cases the speed up and the efficiency of  
the parallel hybrid Mann-extragradient method are $S_p = T_s / T_p \approx 2;  E_p = S_p /2 \approx 1, $ respectively.
\section*{Concluding remarks}
In this paper we proposed three parallel hybrid extragradients methods for finding a common element of the set of solutions 
of equilibrium problems for pseudomonotone bifunctions $\left\{f_i\right\}_{i=1}^N$ and the set of fixed points  of nonexpansive mappings
 $\left\{S_j\right\}_{j=1}^M$ in Hilbert spaces, namely:
\begin{itemize}
\item a parallel hybrid Mann-extragradient method;
\item a parallel hybrid Halpern-extragradient method, and
\item a parallel hybrid iteration-extragradient method. 
\end{itemize}
The efficiency of the proposed parallel algorithms is verified by a simple numerical experiment on computing clusters.
\section*{Acknowledgments} The authors sincerely thank the Editor and anonymous reviewers for their constructive comments which helped 
to improve the quality and presentation of this paper.
We thank Dr. Vu Tien Dzung for performing computation on the LINUX cluster 1350. The research of 
the second and the third authors was partially supported by Vietnam Institute for Advanced Study in 
Mathematics. The third author expresses his gratitude to Vietnam National Foundation for Science and 
Technology Development for a financial support.

\end{document}